\documentclass[abstract,twocolumn,10pt,DIV=18]{scrartcl}

\usepackage{amsmath}
\usepackage{amssymb}
\usepackage{amsthm}
\usepackage{enumitem}
\usepackage[colorlinks = true, linkcolor = blue, citecolor = blue]{hyperref}
\usepackage[nameinlink]{cleveref}
\usepackage{graphicx}

\setcapindent{0pt}
\setkomafont{caption}{\itshape}
\setkomafont{captionlabel}{\bfseries}
\setkomafont{caption}{\footnotesize}

\newtheorem{assumption}{Assumption}
\newtheorem{definition}{Definition}
\newtheorem{example}{Example}
\newtheorem{proposition}{Proposition}
\newtheorem{remark}{Remark}
\newtheorem{theorem}{Theorem}

\crefname{equation}{}{}
\crefname{assumption}{Assumption}{Assumptions}
\crefname{definition}{Definition}{Definitions}
\crefname{proposition}{Proposition}{Propositions}

\begin{document}
\title{Averaging of Random Vibrations in

Mechanical Systems in the Sense of

It{\^o}, Stratonovich, and Sussmann}
\author{Raik Suttner and Christian Ebenbauer \thanks{The authors are with the Chair of Intelligent Control Systems, RWTH Aachen University, Aachen, Germany (e-mail: \{raik.suttner, christian.eben-bauer\}@ic.rwth-aachen.de).} \thanks{This research was supported by the German Research Foundation DFG, project number EB 425/8-1.}}

\date{}
\maketitle

\begin{abstract}
In this paper, we investigate a stochastic averaging principle for a large class of mechanical systems in the presence of random vibrations. We show that a known deterministic averaging principle for affine connection control systems with large-amplitude high-frequency inputs also holds for non-periodic stochastic inputs. The randomly vibrating mechanical system is described by a stochastic differential equation whose solutions do not depend on its interpretation in the sense of It{\^o}, Stratonovich, or Sussmann. We also show that solutions of this stochastic differential equation can be directly computed from a single ordinary differential equation. We illustrate our theoretical results by the example of stochastic source seeking with a nonholonomic vehicle.
\end{abstract}

\section{Introduction}\label{sec:01}
Deterministic vibrational control of mechanical systems is a classical and intensively studied problem \cite{Meerkov1980}, \cite{Bellman1983}, \cite{Levi1999}, \cite{Bullo2002}, \cite{Fidlin2008}, which is motivated by a variety of applications. One can use vibrational control to overpower potentially unstable dynamics. For instance, it is well-know that the Kapitza pendulum \cite{Kapitza1965} can be stabilized in an inverted position through a rapidly vibrating suspension. Similar stabilization effects due to strong high-frequency forcing are also known to occur, more generally, for a large class of mechanical systems \cite{Bellman1986}, \cite{Baillieul1993}, \cite{Taha2020}. Another field of applications for oscillatory inputs with large amplitudes and high frequencies is the problem of trajectory tracking by mechanical control systems \cite{Martinez2003}, \cite{Tahmasian2015}. Recently, such input signals have also been used to design extremum seeking control for mechanical systems \cite{Suttner20234}, \cite{Wang2023}.

The analysis of mechanical systems under vibrational control is often based on suitable averaging techniques \cite{Meerkov1980}, \cite{Weibel1998}, \cite{Vela2002}, \cite{Bullo2002}. Our investigations are motivated by the averaging approach in \cite{Bullo2002} for the class of so-called affine connection control systems. This large class contains, in particular, Lagrangian control systems on Riemannian manifolds with nonholonomic constraints \cite{Lewis1998}. Such systems are typically of control-affine form so that the inputs act as forces into the direction of control vector fields. The inputs are assumed to be given by periodic zero-mean functions with large amplitudes and high frequencies. It is shown in \cite{Bullo2002} that, for sufficiently large amplitudes and frequencies, the system under vibrational control approximates the motion of an averaged system. The averaged system is again an affine connection system, which is driven into directions of so-called \emph{symmetric products} of the control vector fields. One can show (see \cite{Lewis19972}) that symmetric products arise from Lie brackets on the state manifold (i.e. the tangent bundle of the configuration manifold).

To the best of our knowledge, the existing averaging theory is limited to approximations of symmetric products by deterministic vibrational control. One may ask: What happens if \emph{vibrational control} is replaced by \emph{random vibration}? This question is, in particular, motivated by the vast literature on stochastic control for different classes of deterministic systems, which includes linear finite-dimensional systems \cite{Arnold1983}, \cite{Kao1994}, \cite{Arnold1996}, \cite{Kao2000}, non-linear finite-dimensional systems \cite{Mao1994}, \cite{Appleby2008}, \cite{Huang2013}, \cite{Hoshino2016}, and infinite-dimensional systems \cite{Caraballo2004}, \cite{Barbu2011}. For example, the Kapitza pendulum can be stabilized in an inverted position if the suspension is driven by rapid random vibrations \cite{Ovseyevich2006}. Another practical example is the problem of stochastic source seeking, which is inspired by the motion of bacteria toward areas with higher food concentration \cite{Stankovic2009}, \cite{Liu20102}, \cite{Lin20172}. We will address the latter problem in the main part of the paper as an application of our general stochastic averaging result.

A deterministic control system under a stochastic control law is, in general, described by a stochastic differential equation (SDE). Here, we are interested in control by ``large and fast'' random inputs, which is the stochastic analog to deterministic large-amplitude high-frequency inputs. As in the deterministic case, it is natural to investigate the behavior of a randomly vibrating system by means of a suitable averaging analysis. Stochastic versions of the classical Krylov--Bogolyubov averaging principle are well-documented, e.g., in the textbooks \cite{SkorokhodBook}, \cite{FreidlinBook}, \cite{LiuBook}. The key step in our analysis is a suitable change of variables which turns the SDE into an ordinary differential equation (ODE) with random perturbations. Then, we can apply a standard stochastic averaging argument to the randomly perturbed ODE, which in turn allows us to relate the SDE to a deterministic averaged system. Our findings overlap to some degree with the averaging result in \cite{Kao1994} for deterministic linear systems in companion form under parametric noise with large intensity. Indeed, some very simple mechanical systems are linear and of companion form as in \cite{Kao1994}, but in general, mechanical systems are non-linear and their configuration manifold is not a vector space. There is also a certain intersection of our findings with the averaging result in \cite{Meerkov1989} for randomly perturbed ordinary differential equations. However, the analysis in \cite{Meerkov1989} does not cover stochastic differential equations and there is no guarantee that an averaged system actually exists.

The main contributions of this paper are as follows. We generalize the deterministic averaging principle from \cite{Bullo2002} to a stochastic setting. Under suitable assumptions on the employed stochastic input signals, we prove that the solutions of the randomly vibrating system approximate the solutions of an averaged system, which contains symmetric products of the control vector fields. The results obtained here might be useful for future derivations of stochastic controllability and stabilization conditions. We also show that the \emph{It{\^o}--Stratonovich controversy} (see e.g.~\cite{vanKampen1981}) does not arise for the considered class of stochastic affine connection systems. In addition, we provide two path-wise ODE interpretations of the SDE, including the one by Sussmann \cite{Sussmann19782}, and show that all of these interpretations coincide. As a practical application, we present the first stochastic source seeking method for the dynamic (force and torque-controlled) unicycle model.

\section{Motivation: Averaging of Affine Connection Systems with Deterministic Vibrations}\label{sec:02}
In this section, we recall the deterministic averaging principle from \cite{Bullo2002}, which is also extensively explained in the textbook \cite{BulloBook}. Let $\text{\sffamily{\upshape{Q}}}$ be a (topologically) closed submanifold of Euclidean space and let $\nabla$ be an affine connection on $\text{\sffamily{\upshape{Q}}}$. Let $Y_0$, $Y_1$, $\ldots$, $Y_m$ be smooth vector fields on $\text{\sffamily{\upshape{Q}}}$ and let $R$ be a bundle map over the identity of $\text{\sffamily{\upshape{Q}}}$. Let $u\colon[\vphantom{]}0,\infty\vphantom{(})\to\mathbb{R}^m$ be continuously differentiable. For every $\varepsilon>0$, define $u_\varepsilon\colon[\vphantom{]}0,\infty\vphantom{(})\to\mathbb{R}^m$ by $u_\varepsilon(t):=u(t/\varepsilon)$, and let $u_\varepsilon^1,\ldots,u_\varepsilon^m$ denote the component functions of $u_\varepsilon$. Fix an arbitrary element $x_0=(q_0,v_0)$ of the tangent bundle $\text{\sffamily{\upshape{TQ}}}$. For every $\varepsilon>0$, we consider the affine connection system%
\begin{subequations}\label{eq:01}%
\begin{align}
\nabla_D\dot{q}(t) & \ = \ Y_0(q(t)) + R(q(t))\dot{q}(t) \label{eq:01:a} \\
& \qquad + \sum_{i=1}^mY_i(q(t))\,\dot{u}_\varepsilon^i(t) \label{eq:01:b} \\
\intertext{on $\text{\sffamily{\upshape{Q}}}$ with initial condition}
q(0) & \ = \ q_0, \qquad \dot{q}(0) \ = \ v_0. \label{eq:01:c}
\end{align}%
\end{subequations}%
If $u$ is periodic with zero mean, then, for small $\varepsilon>0$, each function $\dot{u}_\varepsilon^i$ acts as a large-amplitude high-frequency input along the direction of the respective control vector field $Y_i$. The second-order differential equation \cref{eq:01:a,eq:01:b} on the configuration manifold $\text{\sffamily{\upshape{Q}}}$ with initial condition \cref{eq:01:c} is equivalent to the first-order differential equation%
\begin{subequations}\label{eq:02}%
\begin{align}
\dot{x}(t) & \ = \ S(x(t)) + Y_0^{\text{v}}(x(t)) + R^{\text{v}}(x(t)) \label{eq:02:a} \\
& \qquad + \sum_{i=1}^mY_i^{\text{v}}(x(t))\,\dot{u}_\varepsilon^i(t) \label{eq:02:b} \\
\intertext{on the state manifold $\text{\sffamily{\upshape{TQ}}}$ with initial condition}
x(0) & \ = \ x_0 \ = \ (q_0,v_0), \label{eq:02:c}
\end{align}%
\end{subequations}%
where $S$ is the geodesic spray for $\nabla$ and $Y_1^{\text{v}},\ldots,Y_m^{\text{v}},R^{\text{v}}$ are the vertical lifts of $Y_1,\ldots,Y_m,R$ (see \cite{BulloBook}). The solution $x$ of \cref{eq:02} is of the form $x=(q,\dot{q})$, where $q$ is the solution of \cref{eq:01}.

Next, we will perform a suitable change of variables that turns \cref{eq:02} into a form in which standard averaging results can be applied. To this end, for every $\mathbb{R}^m$-valued curve $w$ with component functions $w^1,\ldots,w^m$, define a time-varying vector field $\Xi_w$ on $Q$ by%
\begin{equation}\label{eq:03}
\Xi_w(t,q) \ := \ \sum_{i=1}^mY_i(q)\,w^i(t).
\end{equation}%
Given any solution $x=(q,\dot{q})$ of \cref{eq:02}, we can define another curve $\tilde{x}$ in $\text{\sffamily{\upshape{TQ}}}$ by%
\begin{equation}\label{eq:04}
\tilde{x}(t) \ := \ \big(q(t),\dot{q}(t)-\Xi_u(t/\varepsilon,q(t))\big).
\end{equation}%
Notice that $\tilde{x}=(\tilde{q},\tilde{v})$ is not necessarily of the form $\tilde{x}=(\tilde{q},\dot{\tilde{q}})$. It is shown in \cite{Bullo2002} that $\tilde{x}$ is a solution of the initial value problem%
\begin{subequations}\label{eq:05}%
\begin{align}
\dot{\tilde{x}}(t) & \ = \ S(\tilde{x}(t)) + Y_0^{\text{v}}(\tilde{x}(t)) + R^{\text{v}}(\tilde{x}(t)) \label{eq:05:a} \\
& \qquad + \sum_{i=1}^mu_\varepsilon^i(t)\,\big[Y_i^{\text{v}},S+R^{\text{v}}\big](\tilde{x}(t)) \label{eq:05:b} \\
& \qquad - \frac{1}{2}\sum_{i,j=1}^mu_\varepsilon^i(t)\,u_\varepsilon^j(t)\,\langle{Y_i\,\text{:}\,Y_j}\rangle^{\text{v}}(\tilde{x}(t)), \label{eq:05:c} \\
\tilde{x}(0) & \ = \ \big(q_0, v_0-\Xi_u(0,q_0)\big), \label{eq:05:d}
\end{align}%
\end{subequations}%
where $\langle\cdot\,\text{:}\,\cdot\rangle$ denotes the symmetric product with respect to $\nabla$.%
\begin{remark}
The change of variables \cref{eq:04} will be the key step in our approach to stochastic affine connection systems. Notice that the transformed initial value problem \cref{eq:05} has a well-defined unique solution even if the $\mathbb{R}^m$-valued function $u$ is not differentiable but merely continuous. Thus, \cref{eq:05} has a well-defined path-wise solution if $u$ is replaced by an $\mathbb{R}^m$-valued stochastic process with almost surely continuous sample paths. The expression on the right-hand side of \cref{eq:05:a}-\cref{eq:05:c} is given by an infinite sum of Lie brackets, but, as it is shown in \cite{Bullo2002}, all but finitely many of these Lie brackets vanish. In particular, the symmetric products in \cref{eq:05:c} arise from iterated Lie brackets of the geodesic spray $S$ for $\nabla$ and the vertical lifts of the control vector fields $Y_1,\ldots,Y_m$.
\end{remark}%
A standard averaging argument can be applied to \cref{eq:05} if the deterministic oscillation $u$ has the following properties.%
\begin{assumption}[periodicity]\label{ass:Bullo}
The continuously differentiable function $u$ is periodic and satisfies the zero-mean condition%
\begin{equation}\label{eq:06}
\lim_{T\to\infty}\frac{1}{T}\int_0^Tu(\tau)\,\mathrm{d}\tau \ = \ 0.
\end{equation}%
\end{assumption}%
\begin{example}\label{exmp:Bullo}
Let $B^1,\ldots,B^m$ be pairwise distinct positive integers. If, for every $i\in\{1,\ldots,m\}$, the function $u^i$ is given by $u^i(t)=\sin(B^i{t})$, then \Cref{ass:Bullo} is satisfied. Later, in \Cref{exmp:2}, we will replace $(B^1t,\ldots,B^mt)$ by Brownian motion $(B^1(t),\ldots,B^m(t))$ in $\mathbb{R}^m$.
\end{example}%
Suppose that \Cref{ass:Bullo} is satisfied. Then, we can define the $(m\times{m})$-matrix%
\begin{equation}\label{eq:07}
\Lambda \ := \ \frac{1}{2}\lim_{T\to\infty}\frac{1}{T}\int_0^Tu(\tau)\,u(\tau)^\top\,\mathrm{d}\tau
\end{equation}%
whose components are denoted by $\Lambda^{ij}$. In this notation, we can introduce the \emph{averaged system}%
\begin{subequations}\label{eq:08}%
\begin{align}
\dot{\bar{x}}(t) & \ = \ S(\bar{x}(t)) + Y_0^{\text{v}}(\bar{x}(t)) + R^{\text{v}}(\bar{x}(t)) \\
& \qquad- \sum_{i,j=1}^m\Lambda^{ij}\,\langle{Y_i\,\text{:}\,Y_j}\rangle^{\text{v}}(\bar{x}(t)) \\
\intertext{on $\text{\sffamily{\upshape{TQ}}}$ with initial condition}
\bar{x}(0) & \ = \ \big(q_0, v_0-\Xi_u(0,q_0)\big).
\end{align}%
\end{subequations}%
The solution $\bar{x}$ of \cref{eq:08} is of the form $\bar{x}=(\bar{q},\dot{\bar{q}})$, where $\bar{q}$ is the solution of the affine connection system%
\begin{subequations}\label{eq:09}%
\begin{align}
\nabla_D\dot{\bar{q}}(t) & \ = \ Y_0(\bar{q}(t)) + R(\bar{q}(t))\dot{\bar{q}}(t) \label{eq:09:a} \\
& \qquad - \sum_{i,j=1}^m\Lambda^{ij}\,\langle{Y_i\,\text{:}\,Y_j}\rangle(\bar{q}(t)) \label{eq:09:b} \\
\intertext{on $\text{\sffamily{\upshape{Q}}}$ with initial condition}
\bar{q}(0) & \ = \ q_0, \qquad \dot{\bar{q}}(0) \ = \ v_0 - \Xi_u(0,q_0). \label{eq:09:c}
\end{align}%
\end{subequations}%
Theorem~4.1 in \cite{Bullo2002} states that the solution of \cref{eq:01} in the variables \cref{eq:04} approximate the solution of \cref{eq:09} in the limit $\varepsilon\to0$. Moreover, stability properties of the averaged system carry over to the approximating oscillatory system. One of our goals in this paper is to generalize Theorem~4.1 in \cite{Bullo2002} to the case in which the vibrations are generated by a stochastic process with possibly non-differentiable paths.

\section{Affine Connection Systems with Random Vibrations}\label{sec:03}
Now we replace the driving input $u_\varepsilon$ in the deterministic affine connection system~\cref{eq:01:a,eq:01:b} by a stochastic process $U_\varepsilon$. To this end, we adopt the following terminology and conventions from the textbook \cite{HsuBook}. We work with a probability space $(\Omega,\mathcal{F},P)$ equipped with a filtration $\mathcal{F}_\ast=(\mathcal{F}_t)_{t\geq0}$ of $\mathcal{F}$. As usual, it assumed that $\mathcal{F}_0$ contains all $P$-null sets and that $\mathcal{F}_\ast$ is right-continuous. Let $U\colon[\vphantom{]}0,\infty\vphantom{(})\times\Omega\to\mathbb{R}^m$ be a continuous $\mathcal{F}_\ast$-semimartingale (see \cite{HsuBook}). The semimartingale $U$ is the stochastic generalization of the deterministic function $u$ in \Cref{sec:02}. The component functions of $U$ are denoted by $U^1$, $\ldots$, $U^m$. For every $\varepsilon>0$ and every $i\in\{1,\ldots,m\}$, define a continuous real-valued semimartingale $U_\varepsilon^i$ by $U_\varepsilon^i(t,\omega):=U^i(t/\varepsilon,\omega)$. Then, in particular, for every $t\geq0$, the function $U_\varepsilon^i(t)\colon\Omega\to\mathbb{R}$, $\omega\mapsto{U_\varepsilon^i(t,\omega)}$ is an $\mathcal{F}_{t/\varepsilon}$-random variable. More generally, for any stochastic process $Z$ on $(\Omega,\mathcal{F},P)$ and every $t\geq0$, we let $Z(t)$ denote the random variable $\omega\mapsto{Z(t,\omega)}$. Finally, we fix an arbitrary $\text{\sffamily{\upshape{TQ}}}$-valued $\mathcal{F}_0$-random variable $X_0=(Q_0,V_0)$.

In view of the deterministic system~\cref{eq:01}, for every $\varepsilon>0$, we consider the purely formal expression%
\begin{subequations}\label{eq:10}%
\begin{align}
\nabla_D\dot{Q}(t) & \ = \ Y_0(Q(t)) + R(Q(t))\dot{Q}(t) \label{eq:10:a} \\
& \qquad + \sum_{i=1}^mY_i(Q(t))\,\dot{U}_\varepsilon^i(t), \label{eq:10:b} \\
Q(0) & \ = \ Q_0, \qquad V(0) \ = \ V_0, \label{eq:10:c}
\end{align}%
\end{subequations}%
where $\nabla$, $Y_0,Y_1,\ldots,Y_m$, and $R$ are the same as in \Cref{sec:02}. At the moment, the formal expression \cref{eq:10} has no mathematical meaning. In general, it is not admissible to interpret the expression $\dot{U}_\varepsilon^i(t)$ in \cref{eq:10:b} as an ordinary time derivative since the paths of a continuous semimartingale are not necessarily continuously differentiable. Because of the equivalence of~\cref{eq:01} and~\cref{eq:02}, it is natural to regard \cref{eq:10} as (the still purely formal expression)%
\begin{subequations}\label{eq:11}%
\begin{align}
\dot{X}(t) & \ = \ S(X(t)) + Y_0^{\text{v}}(X(t)) + R^{\text{v}}(X(t)) \label{eq:11:a} \\
& \qquad + \sum_{i=1}^mY_i^{\text{v}}(X(t))\,\dot{U}_\varepsilon^i(t), \label{eq:11:b} \\
X(0) & \ = \ X_0. \label{eq:11:c}
\end{align}%
\end{subequations}%
In the next section, we will interpret \cref{eq:11} in four different ways and then show that all four interpretations are actually the same. This will allow us to provide a precise definition of a solution of \cref{eq:10}.

\section{Four Interpretations and Their Equivalence}\label{sec:04}
\subsection{Two Stochastic Integral Interpretations}
Since the state space $\text{\sffamily{\upshape{TQ}}}$ is a manifold, it is natural to interpret~\cref{eq:11} as the stochastic initial value problem%
\begin{subequations}\label{eq:12}%
\begin{align}
\mathrm{d}X(t) & \ = \ \big(S(X(t)) + Y_0^{\text{v}}(X(t)) + R^{\text{v}}(X(t))\big)\,\mathrm{d}t \label{eq:12:a} \\
& \qquad + \sum_{i=1}^mY_i^{\text{v}}(X(t))\circ\mathrm{d}U_\varepsilon^i(t), \label{eq:12:b} \\
X(0) & \ = \ X_0 \label{eq:12:c}
\end{align}%
\end{subequations}%
in the sense of Stratonovich. This first interpretation allows us to define a solution of \cref{eq:11} as follows.%
\begin{definition}\label{def:stochasticSolutionStratonovich}
Let $C^\infty(\text{\sffamily{\upshape{TQ}}})$ denote the set of all smooth real-valued functions on $\text{\sffamily{\upshape{TQ}}}$. A \emph{solution of \cref{eq:11} in the sense of Stratonovich} is a solution of the stochastic initial value problem \cref{eq:12}; that is (see \cite[Definition~1.2.3]{HsuBook}), a continuous $\text{\sffamily{\upshape{TQ}}}$-valued $\mathcal{F}_\ast$-semimartingale $X$ up to an $\mathcal{F}_\ast$-stopping time $\tau$ such that, for every $f\in{C^\infty(\text{\sffamily{\upshape{TQ}}})}$, the real-valued process $f(X)$ satisfies the stochastic integral equation%
\begin{subequations}\label{eq:13}%
\begin{align}
& f(X(t)) \ = \ f(X_0) \label{eq:13:a} \\
& + \int_0^t\big(Sf(X(s)) + Y_0^{\text{v}}f(X(s)) + R^{\text{v}}f(X(s))\big)\,\mathrm{d}s \label{eq:13:b} \\
& + \sum_{i=1}^m\int_0^tY_i^{\text{v}}f(X(s))\circ\mathrm{d}U_\varepsilon^i(s), \qquad 0\leq{t}<\tau, \label{eq:13:c}
\end{align}%
\end{subequations}%
where the $m$ stochastic integrals in \cref{eq:13:c} are understood in the sense of Stratonovich.
\end{definition}%
The following (standard) existence and uniqueness result confirms that \Cref{def:stochasticSolutionStratonovich} makes sense.%
\begin{proposition}[{\cite[Theorem~1.2.9]{HsuBook}}]\label{prop:existenceUniquenessAndInvariance}
There exists a unique solution of \cref{eq:11} in the sense of Stratonovich up to its explosion time.
\end{proposition}%
In general, a stochastic differential equation on a manifold does not have well-defined solutions in the sense It{\^o}. Consequently, one cannot expect that \cref{eq:13} also holds for every $f\in{C^\infty(\text{\sffamily{\upshape{TQ}}})}$ if the stochastic integrals are interpreted in the sense of It{\^o}. However, due to second-order nature of \cref{eq:11}, one can show that an It{\^o} interpretation of \cref{eq:11} nevertheless makes sense for a large class of test functions.%
\begin{definition}\label{def:stochasticSolutionIto}
Let $C^\infty_{\text{It{\^o}}}(\text{\sffamily{\upshape{TQ}}})$ be the set of all $f\in{C^\infty(\text{\sffamily{\upshape{TQ}}})}$ of the form%
\begin{equation}\label{eq:14}
f(q,v) \ = \ g(q) + h(q)v,
\end{equation}%
where $g$ is a smooth real-valued function on $\text{\sffamily{\upshape{Q}}}$ and $h$ is a smooth covector field on $\text{\sffamily{\upshape{Q}}}$. By a \emph{solution of \cref{eq:11} in the sense of It{\^o}}, we mean a continuous $\text{\sffamily{\upshape{TQ}}}$-valued $\mathcal{F}_\ast$-semimartingale $X$ up to an $\mathcal{F}_\ast$-stopping time $\tau$ such that, for every $f\in{C^\infty_{\text{It{\^o}}}(\text{\sffamily{\upshape{TQ}}})}$, the real-valued process $f(X)$ satisfies the stochastic integral equation%
\begin{subequations}\label{eq:15}%
\begin{align}
& f(X(t)) \ = \ f(X_0) \label{eq:15:a} \\
& + \int_0^t\big(Sf(X(s)) + Y_0^{\text{v}}f(X(s)) + R^{\text{v}}f(X(s))\big)\,\mathrm{d}s \label{eq:15:b} \\
& + \sum_{i=1}^m\int_0^tY_i^{\text{v}}f(X(s))\,\mathrm{d}U_\varepsilon^i(s), \qquad 0\leq{t}<\tau, \label{eq:15:c}
\end{align}%
\end{subequations}%
where the $m$ stochastic integrals in \cref{eq:15:c} are understood in the sense of It{\^o}.
\end{definition}%
The next proposition states that the \emph{It{\^o}--Stratonovich controversy} (see e.g.~\cite{vanKampen1981}) does not arise for the interpretations in Definitions~\ref{def:stochasticSolutionStratonovich} and~\ref{def:stochasticSolutionIto}.%
\begin{proposition}\label{prop:ItoStratonovich}
There exists a unique solution of \cref{eq:11} in the sense of It{\^o} up to its explosion time and this solution coincides with the unique solution of \cref{eq:11} in the sense of Stratonovich up to its explosion time.
\end{proposition}%
\begin{proof}
\emph{Existence.} By \Cref{prop:existenceUniquenessAndInvariance}, there exists a solution of \cref{eq:11} in the sense of Stratonovich up to its explosion time. Fix an arbitrary $f\in{C^\infty_{\text{It{\^o}}}(\text{\sffamily{\upshape{TQ}}})}$. An application of the It{\^o} formula (see e.g.~\cite{HsuBook}) shows that the real-valued process $f(X)$ satisfies the It{\^o} integral equation%
\begin{subequations}\label{eq:16}%
\begin{align}
& f(X(t)) \ = \ f(X_0) \label{eq:16:a} \\
& + \int_0^t\big(Sf(X(s)) + Y_0^{\text{v}}f(X(s)) + R^{\text{v}}f(X(s))\big)\,\mathrm{d}t \label{eq:16:b} \\
& + \sum_{i=1}^m\int_0^tY_i^{\text{v}}f(X(s))\,\mathrm{d}U_\varepsilon^i(s) \label{eq:16:c} \\
& + \frac{1}{2}\sum_{i,j=1}^m\int_0^tY_i^{\text{v}}Y_j^{\text{v}}f(X(s))\,\mathrm{d}\langle{U_\varepsilon^i,U_\varepsilon^j}\rangle(s). \label{eq:16:d}
\end{align}%
\end{subequations}%
Since $Y_1^{\text{v}},\ldots,Y_m^{\text{v}}$ are vertical lifts of vector fields and since $f$ is of the form \cref{eq:14}, all iterated Lie derivatives $Y_i^{\text{v}}Y_j^{\text{v}}f$ in \cref{eq:16:d} are identically equal to zero. This shows that $X$ is a solution of \cref{eq:11} in the sense of It{\^o} up to its explosion time.

\emph{Uniqueness.} Let $X$ be a solution of \cref{eq:11} in the sense of It{\^o} up to its explosion time. Recall that $\text{\sffamily{\upshape{Q}}}$ is assumed to be a (topologically) closed submanifold of Euclidean space, say $\mathbb{R}^N$. Let $\iota\colon\text{\sffamily{\upshape{TQ}}}\hookrightarrow\mathbb{R}^N\times\mathbb{R}^N$ denote the natural inclusion map. Notice that each component function $f$ of $\iota$ is an element of $C^\infty_{\text{It{\^o}}}(\text{\sffamily{\upshape{TQ}}})$, and therefore the It{\^o} correction term \cref{eq:16:d} vanishes. This in turn implies that, for each component function $f$ of $\iota$, the real-valued process $f(X)$ satisfies the Stratonovich integral equation \cref{eq:13}. It then follows from Proposition~1.2.7 in \cite{HsuBook} that $X$ is a solution of \cref{eq:11} in the sense of Stratonovich up to its explosion time. By \Cref{prop:existenceUniquenessAndInvariance}, this proves uniqueness of solutions.
\end{proof}%
\begin{remark}
A \emph{natural chart for $\text{\sffamily{\upshape{TQ}}}$} is a chart for $\text{\sffamily{\upshape{TQ}}}$ of the form $(\pi^{-1}(U),\bar{\varphi})$, where $\pi\colon{\text{\sffamily{\upshape{TQ}}}}\to\text{\sffamily{\upshape{Q}}}$ is the canonical projection, $U$ is the domain of a chart $(U,\varphi)$ for $\text{\sffamily{\upshape{Q}}}$, and $\bar{\varphi}$ is given by%
\begin{equation*}
\bar{\varphi}(q,v) \ = \ \big(\varphi^1(q),\ldots,\varphi^n(q),\mathrm{d}\varphi^1(q)v,\ldots,\mathrm{d}\varphi^n(q)v\big).
\end{equation*}%
Notice that each component function of a natural chart for $\text{\sffamily{\upshape{TQ}}}$ is of the form \cref{eq:14}. Consequently, for local representations of \cref{eq:11} in natural charts, it does not matter whether the stochastic differential is interpreted in the sense of It{\^o} or in the sense of Stratonovich. This means that the It{\^o}-Stratonovich controversy does also not arise for representations of \cref{eq:11} in natural charts. Moreover, as it is shown in the proof of \Cref{prop:ItoStratonovich}, there is also no difference between the two interpretations if \cref{eq:11} is considered as a stochastic initial value problem in the ambient Euclidean space of $\text{\sffamily{\upshape{TQ}}}$. Thus, the notion of solution does not depend on the embedding into Euclidean space.
\end{remark}%

\subsection{Two Path-wise Interpretations}
A stochastic integral is, by definition, the limit of a sequence of stochastic processes, which converges in probability. Such a limit does not necessarily exist path-wise. However, in the special case of \cref{eq:11}, we will see that it is actually possible to compute solutions path-wise. It turns out that the stochastic integral equation in \Cref{def:stochasticSolutionStratonovich} in equivalent to an $\Omega$-parameter-dependent ordinary initial value problem. To arrive at the ordinary differential equation, we will apply the change of variables \cref{eq:04} from \cite{Bullo2002} path-wise. Let $X$ be a solution of \cref{eq:11} in the sense of Stratonovich; that is, $X$ solves the stochastic initial value problem~\cref{eq:12}. It can be easily seen from a local chart representation of \cref{eq:12:a,eq:12:b} that $X$ is of the form $X=(Q,\dot{Q})$ with an (almost surely) continuously differentiable $\text{\sffamily{\upshape{Q}}}$-valued $\mathcal{F}_\ast$-semimartingale $Q$, where $\dot{Q}$ denotes the path-wise time derivative. For this reason, we may apply the change of variables \cref{eq:04} to $X=(Q,\dot{Q})$ path-wise. Then%
\begin{equation}\label{eq:17}
\tilde{X}(t) \ := \  \big(Q(t),\dot{Q}(t)-\Xi_U(t/\varepsilon,Q(t))\big)
\end{equation}%
defines another $\text{\sffamily{\upshape{TQ}}}$-valued $\mathcal{F}_\ast$-semimartingale $\tilde{X}$. Since the rules of ordinary calculus and Stratonovich calculus are the same, we can perform the same computations as in the proof of Theorem~4.1 in~\cite{Bullo2002} and obtain that $\tilde{X}$ is (almost surely) continuously differentiable and solves the initial value problem%
\begin{subequations}\label{eq:18}%
\begin{align}
\mathrm{d}\tilde{X}(t) & \ = \ \big(S(\tilde{X}(t)) + Y_0^{\text{v}}(\tilde{X}(t)) + R^{\text{v}}(\tilde{X}(t))\big)\,\mathrm{d}t \label{eq:18:a} \\
& \qquad + \sum_{i=1}^mU_\varepsilon^i(t)\,\big[Y_i^{\text{v}},S+R^{\text{v}}\big](\tilde{X}(t))\,\mathrm{d}t \label{eq:18:b} \\
& \qquad - \frac{1}{2}\sum_{i,j=1}^mU_\varepsilon^i(t)\,U_\varepsilon^j(t)\,\langle{Y_i\,\text{:}\,Y_j}\rangle^{\text{v}}(\tilde{X}(t))\mathrm{d}t, \label{eq:18:e} \\
\tilde{X}(0) & \ = \ \big(Q_0, V_0-\Xi_U(0,Q_0)\big), \label{eq:18:f}
\end{align}%
\end{subequations}%
which is the stochastic generalization of \cref{eq:05}. Notice that the right-hand side of \Cref{eq:18} only involves the deterministic differential ``$\mathrm{d}t$.'' Since the sample paths of $U_\varepsilon$ are almost surely continuous, it does not matter whether \Cref{eq:18} is interpreted as a stochastic integral equation or path-wise as an ordinary initial value problem. Therefore, a solution of \cref{eq:18} can be determined in the sense of ordinary differential equations (ODEs), which motivates the next definition.%
\begin{definition}\label{def:stochasticSolutionBullo}
By a \emph{solution of \cref{eq:11} in the sense of ODEs}, we mean a continuous $\text{\sffamily{\upshape{TQ}}}$-valued $\mathcal{F}_\ast$-semimartingale $X$ up to an $\mathcal{F}_\ast$-stopping time $\tau$ for which there exists a solution $\tilde{X}=(\tilde{Q},\tilde{V})$ of \cref{eq:18} on the stochastic time interval $[\vphantom{]}0,\tau\vphantom{(})$ such that%
\begin{equation}\label{eq:19}
X(t) \ = \ \big(\tilde{Q}(t),\tilde{V}(t)+\Xi_U(t/\varepsilon,\tilde{Q}(t))\big), \qquad 0\leq{t}<\tau.
\end{equation}%
\end{definition}%
\begin{proposition}\label{prop:BulloStratonovich}
There exists a unique solution of \cref{eq:11} in the sense of ODEs up to its explosion time and this solution coincides with the unique solution of \cref{eq:11} in the sense of Stratonovich up to its explosion time.
\end{proposition}%
\begin{proof}
The same computations as in the proof of Theorem~4.1 in~\cite{Bullo2002} show that, for every solution $X$ of \cref{eq:12}, a solution $\tilde{X}$ of \cref{eq:18} is given by \cref{eq:17}. Conversely, for every solution $\tilde{X}$ of \cref{eq:18}, a solution $X$ of \cref{eq:12} is given by \cref{eq:19}. Now the claim follows from the fact that both \cref{eq:12} and \cref{eq:18} have the existence and uniqueness property of solutions up their explosion times (see \cite[Theorem~1.2.9]{HsuBook}).%
\end{proof}%
Finally, to provide a fourth interpretation of \cref{eq:11}, we recall Sussmann's path-wise definition of solutions of stochastic differential equations from \cite{Sussmann19782}. For every $T>0$ and every metrizable topological space $M$, let $C^0([0,T],M)$ denote the set of continuous maps from $[0,T]$ into $M$ together with topology of uniform convergence.

\emph{Definition (\cite[Definition~1]{Sussmann19782}):}
Let $\omega\in\Omega$ be a sample for which the $\omega$-sample path $u$ of $U_\varepsilon$ is continuous. Let $\tau>0$. An \emph{$\omega$-sample path solution of \cref{eq:11} in the sense of Sussmann up to time $\tau$} is a curve $x\colon[\vphantom{]}0,\tau\vphantom{(})\to{\text{\sffamily{\upshape{TQ}}}}$ so that, for every $T\in(0,\tau)$, the following holds: There exists a continuous map $\Gamma$ from a neighborhood $\mathcal{N}$ of $u|_{[0,T]}$ in $C^0([0,T],\mathbb{R}^m)$ into $C^0([0,T],\text{\sffamily{\upshape{TQ}}})$ such that $\Gamma(u|_{[0,T]})=x|_{[0,T]}$ and, for every $w\in\mathcal{N}$ of class $C^1$, the curve $\Gamma(w)\colon[0,T]\to{\text{\sffamily{\upshape{TQ}}}}$ is of class $C^1$ and solves the ordinary initial value problem%
\begin{subequations}\label{eq:20}%
\begin{align}
\dot{x}(t) & \ = \ S(x(t)) + Y_0^{\text{v}}(x(t)) + R^{\text{v}}(x(t)) \\
& \qquad + \sum_{i=1}^mY_i^{\text{v}}(x(t))\,\dot{w}^i(t), \\
x(0) & \ = \ X_0(\omega).
\end{align}%
\end{subequations}%

Notice \cref{eq:20} is the same as \cref{eq:02} for $w=u^\varepsilon$ and $X_0(\omega)=x_0$.%
\begin{definition}[{\cite[Definition~2]{Sussmann19782}}]\label{def:stochasticSolutionSussmann}
A \emph{solution of \cref{eq:11} in the sense of Sussmann} is a continuous $\text{\sffamily{\upshape{TQ}}}$-valued $\mathcal{F}_\ast$-semimartingale $X$ up to an $\mathcal{F}_\ast$-stopping time $\tau$ such that, for almost every $\omega\in\Omega$ for which the $\omega$-sample path of $U$ is continuous, the $\omega$-sample path of $X$ is an $\omega$-sample path solution of \cref{eq:11} in the sense of Sussmann up to time $\tau(\omega)$.
\end{definition}%
As it is pointed out in Section~7 of \cite{Sussmann19782}, in general, one cannot expect that a stochastic initial value problem has a solution in the sense of Sussmann if the driving semimartingale has multiple components. However, since the Lie brackets of vertical lifts of vector fields vanish, one can show that a unique solution of \cref{eq:11} in the sense of Sussmann exists.%
\begin{proposition}\label{prop:SussmannBullo}
There exists a unique solution of \cref{eq:11} in the sense of Sussmann up to its explosion time and this solution coincides with the unique solution of \cref{eq:11} in the sense of ODEs up to its explosion time.
\end{proposition}%
\begin{proof}
By \Cref{prop:BulloStratonovich}, there exists a unique solution $X$ of \cref{eq:11} in the sense of ODEs up to its explosion time $\tau$, which is given by the unique solution $\tilde{X}$ of \cref{eq:18} through the change of variables \cref{eq:19}. Fix an arbitrary sample $\omega\in\Omega$ for which the $\omega$-sample path $u$ of $U_\varepsilon$ is continuous. The proof is complete if we can show that there exists a unique $\omega$-sample path solution of \cref{eq:11} in the sense of Sussmann up to its explosion time and that this solution coincides with the $\omega$-sample path $x$ of $X$. To this end, fix an arbitrary $T\in(0,\tau(\omega))$. For every $w\in\mathbb{R}^m$, a diffeomorphism $\Phi(w)\colon\text{\sffamily{\upshape{TQ}}}\to\text{\sffamily{\upshape{TQ}}}$ is defined by%
\begin{equation}\label{eq:21}
\Phi(w)(q,v) \ := \ \Big(q,v-\sum_{i=1}^mY_i(q)\,w^i\Big)
\end{equation}%
with inverse map $\Phi(-w)$. The $\omega$-sample path $\tilde{x}$ of $\tilde{X}$ is then given by $\tilde{x}(t)=\Phi(u(t))(x(t))$ for every $t\in[0,T]$. Moreover, $\tilde{x}|_{[0,T]}$ is the unique solution of%
\begin{subequations}\label{eq:22}%
\begin{align}
\dot{\tilde{x}}(t) & \ = \ S(\tilde{x}(t)) + Y_0^{\text{v}}(\tilde{x}(t)) + R^{\text{v}}(\tilde{x}(t)) \label{eq:22:a} \\
& \qquad + \sum_{i=1}^mw^i(t)\,\big[Y_i^{\text{v}},S+R^{\text{v}}\big](\tilde{x}(t)) \label{eq:22:b} \\
& \qquad - \frac{1}{2}\sum_{i,j=1}^mw^i(t)\,w^j(t)\,\langle{Y_i^{\text{v}}\,\text{:}\,Y_j}\rangle^{\text{v}}(\tilde{x}(t)), \label{eq:22:e} \\
\tilde{x}(0) & \ = \ \Phi(w(0))(X_0(\omega)) \label{eq:22:f}
\end{align}%
\end{subequations}%
on $[0,T]$ for $w=u|_{[0,T]}$.

\emph{Existence.} Continuous dependence of solutions of ODEs on (infinite-dimensional) parameters 
guarantees the existence of a sufficiently small neighborhood $\mathcal{N}$ of $u|_{[0,T]}$ in $C^0([0,T],\mathbb{R}^m)$ such that, for every $w\in\mathcal{N}$, there exists a unique solution $\tilde{\Gamma}(w)$ of \cref{eq:22} on $[0,T]$. The so-defined map $\tilde{\Gamma}$ from $\mathcal{N}$ to $C^0([0,T],\text{\sffamily{\upshape{TQ}}})$ is continuous and its values are in fact of class $C^1$. Another continuous map $\Gamma$ from $\mathcal{N}$ to $C^0([0,T],\text{\sffamily{\upshape{TQ}}})$ is defined by%
\begin{equation}\label{eq:23}
\Gamma(w)(t) \ := \ \Phi(-w(t))(\tilde{\Gamma}(t)).
\end{equation}%
Then $\Gamma(u|_{[0,T]})=x|_{[0,T]}$ and, for every $w\in\mathcal{N}$ of class $C^1$, also the curve $\Gamma(w)$ is of class $C^1$ and solves \cref{eq:20}.

\emph{Uniqueness.} Suppose that $\hat{\Gamma}$ is another continuous map from a neighborhood $\hat{\mathcal{N}}$ of $u|_{[0,T]}$ in $C^0([0,T],\mathbb{R}^m)$ into $C^0([0,T],\text{\sffamily{\upshape{TQ}}})$ such that, for every $w\in\hat{\mathcal{N}}$ of class $C^1$, also the curve $\hat{\Gamma}(w)\colon[0,T]\to{\text{\sffamily{\upshape{TQ}}}}$ is of class $C^1$ and solves \cref{eq:20}. Uniqueness of solutions of \cref{eq:20} implies $\Gamma(w)=\hat{\Gamma}(w)$ for every $w\in\mathcal{N}\cap\hat{\mathcal{N}}$ of class $C^1$. Since $\mathcal{N}\cap\hat{\mathcal{N}}$ is a neighborhood of $u|_{[0,T]}$ and since $C^1([0,T],\text{\sffamily{\upshape{TQ}}})$ is dense in $C^0([0,T],\text{\sffamily{\upshape{TQ}}})$, it follows that $x|_{[0,T]}=\hat{\Gamma}(u|_{[0,T]})$.
\end{proof}%
We summarize the statements of Propositions~\ref{prop:existenceUniquenessAndInvariance}, \ref{prop:ItoStratonovich}, \ref{prop:BulloStratonovich}, and \ref{prop:SussmannBullo} in the following result.%
\begin{theorem}\label{thm:2}
There exist unique solutions of \cref{eq:11} in the sense of Stratonovich (\Cref{def:stochasticSolutionStratonovich}), It{\^o} (\Cref{def:stochasticSolutionIto}), ODEs (\Cref{def:stochasticSolutionBullo}), and Sussmann (\Cref{def:stochasticSolutionSussmann}) up to their explosion times, and all of these four solutions coincide (almost surely).
\end{theorem}%
Now we can also make sense of the stochastic affine connection system \cref{eq:10:a,eq:10:b} with initial condition \cref{eq:10:c}.%
\begin{definition}\label{def:5}
The \emph{solution of \cref{eq:10} up to its explosion time} is the (almost surely) continuously differentiable $\text{\sffamily{\upshape{Q}}}$-valued $\mathcal{F}_\ast$-semimartingale $Q_\varepsilon:=\pi(X_\varepsilon)$, where $\pi\colon\text{\sffamily{\upshape{TQ}}}\to\text{\sffamily{\upshape{Q}}}$ is the canonical projection and $X_\varepsilon=(Q_\varepsilon,\dot{Q}_\varepsilon)$ is the unique solution of \cref{eq:11} up to its explosion time as in \Cref{thm:2}.
\end{definition}%

\section{Averaging of Stochastic Affine Connection Systems}\label{sec:05}
In the previous section, we have seen that the change of variables \cref{eq:17} turns the stochastic initial value problem \cref{eq:11} into the ordinary initial value problem \cref{eq:18}. One can apply standard averaging results (as in~\cite{SkorokhodBook}, \cite{FreidlinBook}, \cite{LiuBook}) to~\cref{eq:18} if the semimartingale $U$ has the following ergodic properties.%
\begin{assumption}[ergodicity]\label{ass:ergodicProperty}
The continuous semimartingale~$U$%
\begin{enumerate}
	\item takes its values in a compact subset $E$ of $\mathbb{R}^m$,
	\item is \emph{ergodic} in the sense that $U$ admits a normalized invariant measure $\mu$ on $E$ such that%
	\begin{equation}\label{eq:24}
	\lim_{T\to\infty}\frac{1}{T}\int_0^Tf(U(s))\,\mathrm{d}s \ = \ \int_Ef(u)\,\mu(\mathrm{d}u)
	\end{equation}%
	for every $\mu$-integrable real-valued function $f$ on $E$, and
	\item its invariant measure $\mu$ satisfies the zero-mean condition%
	\begin{equation}\label{eq:25}
	\int_Eu\,\mu(\mathrm{d}u) \ = \ 0
	\end{equation}%
\end{enumerate}%
(where \emph{normalized} means $\mu(E)=1$).
\end{assumption}%
\begin{remark}
It is known (e.g.~from Theorems~2.1 and~3.1 in \cite{Khasminskii1960}) that if the semimartingale $U$ is a diffusion process and if $U$ takes its values in a compact subset of $\mathbb{R}^m$, then $U$ admits a unique normalized invariant measure, which satisfies \cref{eq:24}.
\end{remark}%
As an alternative to the deterministic vibrations in \Cref{exmp:Bullo} one can consider the following type of random vibrations.%
\begin{example}\label{exmp:2}
Let $B$ be standard Brownian motion in $\mathbb{R}^m$ with independent components $B^1,\ldots,B^m$. Suppose that, for every $i\in\{1,\ldots,m\}$, the $i$th component $U^i$ of the semimartingale $U$ is given by $U^i(t)=\sin(B^i(t))$. Then $U$ is a diffusion process, which takes its values in the cube $E:=[-1,1]^m$ and its normalized invariant measure $\mu$ is given by%
\begin{equation}\label{eq:26}
\mu(A) \ = \ \frac{1}{\pi^m}\int_A\frac{\mathrm{d}u^1}{\sqrt{1-(u^1)^2}}\cdots\frac{\mathrm{d}u^m}{\sqrt{1-(u^m)^2}}
\end{equation}%
for every Borel subset $A$ of $E$. This follows from the fact that the invariant measure of Brownian motion $(\cos(B^i),\sin(B^i))$ on the unit circle is just the normalized Lebesgue measure.
\end{example}%
Suppose that \Cref{ass:ergodicProperty} is satisfied with normalized invariant measure $\mu$ as therein. Then, condition \cref{eq:25} can be viewed as a stochastic generalization of \cref{eq:06}. As a stochastic generalization of \cref{eq:07}, define the $(m\times{m})$-matrix%
\begin{equation}\label{eq:27}
\Lambda \ := \ \frac{1}{2}\int_Euu^\top\mu(\mathrm{d}u)
\end{equation}%
whose components are denoted by $\Lambda^{ij}$. Using \cref{eq:24}, we can compute the long-time average of the right-hand side of \cref{eq:18} and we obtain the \emph{averaged system}%
\begin{subequations}\label{eq:28}%
\begin{align}
\dot{\bar{X}}(t) & \ = \ S(\bar{X}(t)) + Y_0^{\text{v}}(\bar{X}(t)) + R^{\text{v}}(\bar{X}(t)) \label{eq:28:a} \\
& \qquad - \sum_{i,j=1}^m\Lambda^{ij}\,\langle{Y_i\,\text{:}\,Y_j}\rangle^{\text{v}}(\bar{X}(t)) \label{eq:28:b} \\
\intertext{on $\text{\sffamily{\upshape{TQ}}}$ with initial condition}
\bar{X}(0) & \ = \ \big(Q_0, V_0-\Xi_U(0,Q_0)\big), \label{eq:28:c}
\end{align}%
\end{subequations}%
which is the same as \cref{eq:08} except for a potentially non-deterministic initial condition. One can write the first-order differential equation \cref{eq:28:a,eq:28:b} on $\text{\sffamily{\upshape{TQ}}}$ equivalently as a second-order differential equation on $\text{\sffamily{\upshape{Q}}}$, which results in the affine connection system%
\begin{subequations}\label{eq:29}%
\begin{align}
\nabla_D\dot{\bar{Q}}(t) & \ = \ Y_0(\bar{Q}(t)) + R(\bar{Q}(t))\dot{\bar{Q}}(t) \label{eq:29:a} \\
& \qquad - \sum_{i,j=1}^m\Lambda^{ij}\,\langle{Y_i\,\text{:}\,Y_j}\rangle(\bar{Q}(t)) \label{eq:29:b} \\
\intertext{on $\text{\sffamily{\upshape{Q}}}$ with initial condition}
\bar{Q}(0) & \ = \ Q_0, \qquad \dot{\bar{Q}}(0) \ = \ V_0 - \Xi_U(0,Q_0). \label{eq:29:c}
\end{align}%
\end{subequations}%
Let $|\cdot|$ denote the Euclidean norm on the ambient Euclidean space of $\text{\sffamily{\upshape{TQ}}}$. A direct application of Lemma~4.4 and Theorem~4.3 in the textbook \cite{LiuBook} to the transformed initial value problem \cref{eq:18} leads to the following stochastic generalization of Theorem~4.1 in \cite{Bullo2002} in the terminology of \Cref{def:5}.%
\begin{theorem}\label{thm:mainResult}
Suppose that \Cref{ass:ergodicProperty} is satisfied. Assume that $U(0)=0$. Then, for every deterministic $(Q_0,V_0)\in\text{\sffamily{\upshape{TQ}}}$ and every $T>0$, the following implication holds: If the solution $\bar{Q}$ of \cref{eq:29} exists on $[0,T]$, then, the solution $Q_\varepsilon$ of \cref{eq:10} satisfies%
\begin{subequations}\label{eq:30}%
\begin{align}
\lim_{\varepsilon\to0}\sup_{0\leq{t}\leq{T}}\big|\bar{Q}(t) - Q_\varepsilon(t)\big| & \ = \ 0, \\
\lim_{\varepsilon\to0}\sup_{0\leq{t}\leq{T}}\big|\dot{\bar{Q}}(t) + \Xi(t/\varepsilon,Q_\varepsilon(t)) - \dot{Q}_\varepsilon(t)\big| & \ = \ 0
\end{align}%
\end{subequations}%
almost surely. Furthermore, if $(q_\ast,0)\in\text{\sffamily{\upshape{TQ}}}$ is a locally exponentially stable equilibrium point of \cref{eq:29:a,eq:29:b}, then there exist constants $r,c,\gamma>0$ such that, for every deterministic $(Q_0,V_0)\in\text{\sffamily{\upshape{TQ}}}$ with $|Q_0-q_\ast|\leq{r}$, $|V_0|\leq{r}$ and every $\delta>0$, the solution $Q_\varepsilon$ of \cref{eq:10} satisfies%
\begin{subequations}\label{eq:31}%
\begin{align}
\lim_{\varepsilon\to0}\inf\big\{t\geq0\colon \ \ & \big|q_\ast - Q_\varepsilon(t)\big| > \\
& \ c\,|q_\ast - Q_0|\,\mathrm{e}^{-\gamma\,t} + \delta\big\} \ = \ \infty, \\
\lim_{\varepsilon\to0}\inf\big\{t\geq0\colon \ \ & \big|\Xi(t/\varepsilon,Q_\varepsilon(t)) - \dot{Q}_\varepsilon(t)\big| > \\
& c\,|V_0|\,\mathrm{e}^{-\gamma\,t} + \delta\big\} \ = \ \infty
\end{align}%
\end{subequations}%
almost surely.
\end{theorem}%
Thus, the averaging properties of affine connection systems under vibrational control from \cite{Bullo2002} remains valid if the deterministic oscillations are generalized to random vibrations.

\section{Numerical Example: Stochastic Source Seeking with a Nonholonomic Vehicle}\label{sec:06}
In this section, we illustrate the stochastic averaging result, \Cref{thm:mainResult}, by an example; namely a source seeking dynamic unicycle. In this simple example, the configuration manifold $\text{\sffamily{\upshape{Q}}}$ is the special Euclidean group of dimension two, which can be identified with $\mathbb{R}^2\times\mathbb{S}$, where $\mathbb{S}$ denotes the unit circle in $\mathbb{R}^2$. An element $(p,o)$ of $\mathbb{R}^2\times\mathbb{S}$ represents a position $p$ and an orientation $o$. The dynamic unicycle is a mechanical system with mass $m_\shortparallel>0$ and moment of inertia $m_\sphericalangle>0$. Its translational and rotational motion is assumed to be subject to linear velocity dependent damping with damping constants $k_\shortparallel>0$ and $k_\sphericalangle>0$. In standard coordinates $(p^1,p^2,\theta)$ for $\mathbb{R}^2\times\mathbb{S}$, the position is described by a two-component vector $(p^1,p^2)$ and the orientation is described by an angle $\theta$. The unicycle is subject to nonholonomic constraints: It cannot move perpendicular to its current alignment. The remaining translational and rotational velocity components are denoted by $v^\shortparallel$ and $v^\sphericalangle$. In these coordinates, the unicycle model reads%
\begin{subequations}\label{eq:32}%
\begin{align}
\dot{p}^1(t) & \ = \ v^\shortparallel(t)\,\cos(\theta(t)), \\
\dot{p}^2(t) & \ = \ v^\shortparallel(t)\,\sin(\theta(t)), \\
\dot{\theta}(t) & \ = \ v^\sphericalangle(t), \\
m_\shortparallel\,\dot{v}^\shortparallel(t) & \ = \ -k_\shortparallel\,v^\shortparallel(t) + u^\shortparallel(t), \\
m_\sphericalangle\,\dot{v}^\sphericalangle(t) & \ = \ -k_\sphericalangle\,v^\sphericalangle(t) + u^\sphericalangle(t),
\end{align}%
\end{subequations}%
where $u^\shortparallel(t)$ and $u^\sphericalangle(t)$ are real-valued inputs for the translational and rotational motion. One can write \cref{eq:32} more compactly as an affine connection system%
\begin{equation}\label{eq:33}
\nabla_D\dot{q}(t) \ = \ R(q(t))\dot{q}(t) + Y_\shortparallel(q(t))\,u^\shortparallel(t) + Y_\sphericalangle(q(t))\,u^\sphericalangle(t)
\end{equation}%
on $\text{\sffamily{\upshape{Q}}}$, where the bundle map $R$ contains the linear damping terms, the vector fields $Y_\shortparallel$, $Y_\sphericalangle$ describe the control directions, and $\nabla$ is a suitable \emph{constrained connection} (see \cite{BulloBook}), which models the nonholonomic velocity constraint of the unicycle.

The task of the unicycle is to locate the source of an unknown position-dependent scalar signal. This signal is described by an analytically unknown smooth real-valued function $\psi$ on $\mathbb{R}^2$, which is assumed to be sufficiently smooth. Here, ``analytically unknown'' means that the functional dependence of $\psi$ on the position is not assumed to be known. Also the current system state $(q(t),\dot{q}(t))$ of the unicycle is not assumed to be known. Only real-time measurements of the signal are available. To this end, it is assumed that unicycle is equipped with a suitable sensor so that it can measure the value of $\psi$ at the current position of the sensor. The sensor is mounted with positive distance $\rho$ from the vehicle center position $p(t)$ into the direction $o(t)$ of the current alignment. That is, a measurement of $\psi$ in the current configuration $q(t)=(p(t),o(t))$ of the unicycle results in the measured value%
\begin{equation}\label{eq:34}
y(t) \ = \ \psi(p(t)+\rho\,o(t)).
\end{equation}%
We are interested in a feedback law that only requires real-time measurements of \cref{eq:34} and steers the unicycle towards a position where $\psi$ attains a minimum value. Such a method is called a \emph{source seeking method}.%
\begin{figure*}
\includegraphics{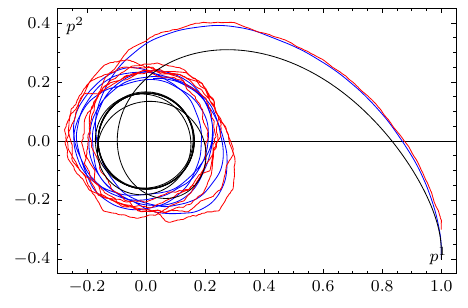} \includegraphics{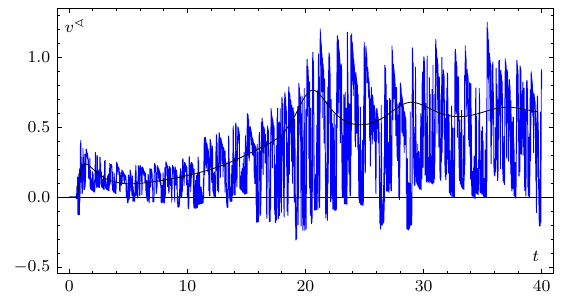}
\caption{Left: Trajectories of the vehicle center of the averaged system (black) and trajectories of the vehicle center (blue) and the sensor (red) of the closed-loop system. Right: Rotational velocity of the averaged system (black) and rotational velocity of the closed-loop system (blue).}
\label{fig:numericalResult}
\end{figure*}%
In the following, we consider a stochastic generalization of the deterministic source seeking method from \cite{Suttner2023}. To this end, let $\alpha\colon\mathbb{R}\to\mathbb{R}$ be a smooth function, which is specified below. Let $\lambda$ and $a$ be positive constants. The stochastic source seeking method from \cite{Suttner2023} for the dynamic unicycle model \cref{eq:32} is given by%
\begin{subequations}\label{eq:35}%
\begin{align}
u^\shortparallel(t) & \ := \ a, \\
u^\sphericalangle(t) & \ := \ \lambda\,\alpha\big(h\,(y(t)-\eta(t))\big)\,\dot{U}_\varepsilon(t),
\end{align}%
\end{subequations}%
where $\varepsilon$ and $h$ are positive control parameters and $\eta(t)$ is the real-valued state of the high-pass filter dynamics%
\begin{equation}\label{eq:36}
\dot{\eta}(t) \ = \ h\,(y(t)-\eta(t))
\end{equation}%
in order to remove a possible offset from the measured signal $y(t)$ in~\cref{eq:34}. The real-valued function $U_\varepsilon$ is a deterministic oscillation with frequency $1/\varepsilon$. It is shown in \cite{Suttner2023} that control law \cref{eq:35} can lead to a circular motion of the unicycle around the source if $\varepsilon>0$ is sufficiently small. Using the stochastic averaging result (\Cref{thm:mainResult}) one can arrive at the same conclusion when the deterministic oscillation is replaced by random vibration. The closed-loop system is then interpreted in the sense of \Cref{def:5}. As in \Cref{exmp:2}, a suitable choice of $U_\varepsilon$ is%
\begin{equation}\label{eq:37}
U_\varepsilon(t) \ := \ \sin(B(t/\varepsilon)),
\end{equation}%
where $B$ is standard Brownian motion on the real line.

To illustrate our stochastic averaging result we simulate system \cref{eq:33} under control law \cref{eq:35} with random vibrations \cref{eq:37}. To this end, we choose the same constants and parameters as in \cite{Suttner2023}. The unknown real-valued function $\psi$ on $\mathbb{R}^2$ is defined by $\psi(p):=-6/(1+|p|^2/5)$. The constants $\rho$ and $a$ are set equal to $0.1$, and the constants $m_\shortparallel$, $m_\sphericalangle$, $k_\shortparallel$, and $k_\sphericalangle$ are set equal $1$. Let $\lambda:=4$ and $h:=10$. The function $\alpha$ is defined by $\alpha(z):=\frac{\sqrt{2}}{h}\,\sqrt{2\,h\,z+\log(2\cosh(2\,h\,z))}$. For this choice of constants and parameters, it is shown in \cite{Suttner2023} that a certain circular motion about the minimizer of $\psi$ is locally exponentially stable for the respective averaged system. A trajectory of this averaged system is indicated by black lines in \Cref{fig:numericalResult}. Because of \Cref{thm:mainResult}, one can expect the same behavior for the approximating closed-loop system if the parameter $\varepsilon>0$ is sufficiently small. We choose $\varepsilon=0.04$. The simulation result in \Cref{fig:numericalResult} is generated for initial position $p(0)=(1.0,-0.4)$, initial orientation $o(0)=(0,1)$, initial velocity $v(0)=0$, and initial filter state $\eta(0)=0$. As predicted by \Cref{thm:mainResult}, the solution of the closed-loop system approximates the solution of the averaged system.

\section{Conclusions}
We have provided an approach to stochastic Lie bracket approximations in which the solutions of a stochastic affine connection system approximate the solutions of a deterministic averaged system, provided that the driving process has suitable ergodic properties. We have shown that the considered class of stochastic affine connection systems can be interpreted in various (seemingly different) ways, but all of these interpretations result in the same solutions. This is in particular true for the established interpretations of stochastic differential equations in the sense of It{\^o}, Stratonovich, and Sussmann.%
\bibliographystyle{abbrv}
\bibliography{bibFile}
\end{document}